\documentclass{siamltex}%
\usepackage{amsfonts}
\usepackage{amsmath}
\usepackage{amssymb}
\usepackage{graphicx}
\usepackage{hyperref}%
\newtheorem{remark}[theorem]{Remark}

\begin{document}

\title{Bounds for Invariance Pressure}
\author{Fritz Colonius\\Institut f\"{u}r Mathematik, Universit\"{a}t Augsburg, Augsburg, Germany
\and Alexandre J. Santana and Jo\~{a}o A. N. Cossich\\Departamento de Matem\'{a}tica, Universidade Estadual de Maring\'{a}\\Maring\'{a}, Brazil}
\maketitle

\begin{center}
\today

\end{center}

\textbf{Abstract: }This paper provides an upper for the invariance pressure of
control sets with nonempty interior and a lower bound for sets with finite
volume. In the special case of the control set of a hyperbolic linear control
system in $\mathbb{R}^{d}$ this yields an explicit formula. Further
applications to linear control systems on Lie groups and to inner control sets
are discussed.


\textbf{AMS\ subject classification. 93C15, 37B40, 94A17}

\section{Introduction}

The notion of invariance pressure generalizes invariance entropy by adding
potentials $f$ on the control range. It has been introduced and analyzed in
Colonius, Cossich, Santana \cite{Cocosa, Cocosa2}. Zhong and Huang \cite{ZH18}
show that invariance pressure can be characterized as a dimension-like notion
within the framework due to Pesin. A basic reference for invariance entropy is
Kawan's monograph \cite{Kawa13}; here also the relation to minimal data rates
is explained which gives the main motivation from applications. Further
references include the seminal paper Nair, Evans, Mareels and Moran
\cite{NEMM04} as well as Colonius and Kawan \cite{ColKa09a} and da Silva and
Kawan \cite{AdrKa}, \cite{daSilK18}. In the latter paper, robustness
properties in the hyperbolic case are proved. Huang and Zhong \cite{HuanZ18}
show that several generalized notions of invariance entropy fit into the
dimension-theoretic framework due to Pesin.

The main results of the present paper are upper and lower bounds for the
invariance pressure of compact subsets $K$ in a control set $D$ with nonvoid
interior and compact closure. For hyperbolic linear control systems in
$\mathbb{R}^{d}$ this yield a formula for the invariance pressure. We also
give applications for inner control sets and for certain linear systems on Lie
groups. Invariance entropy of these systems has been analyzed by da Silva
\cite{daSil14}.

Section \ref{Section2} collects results on linearization of control systems
and on the notion of invariance pressure. Upper and lower bounds for
invariance pressure are given in Sections \ref{Section3} and \ref{Section4},
respectively. Section \ref{Section5} presents a formula for the invariance
pressure of linear control systems in $\mathbb{R}^{d}$ and Section
\ref{Section6} discusses applications linear systems on Lie groups and for
inner control sets.

\section{Preliminaries\label{Section2}}

In this section we first recall basic notions for control systems on manifolds
and their linearization \ Then the concepts of invariance pressure and outer
invariance pressure are presented as well as some of their properties.

\subsection{Control systems and linearization}

Throughout the paper, $M$ will denote a smooth manifold, that is, a connected,
second-countable, topological Hausdorff manifold endowed with a $C^{\infty}$
differentiable structure. A continuous-time \textbf{\textit{control system}}
on a smooth manifold $M$ is a family of ordinary differential equations
\begin{equation}
\ \dot{x}(t)=F(x(t),\omega(t)),\omega\in\mathcal{U},\label{2.1}%
\end{equation}
on $M$ which is parametrized by measurable functions $\omega:\mathbb{R}%
\rightarrow\mathbb{R}^{m}$, $\omega(t)\in U\subset\mathbb{R}^{m}$ almost
everywhere, called \textbf{\textit{controls}} forming the set $\mathcal{U}$ of
\textbf{\textit{admissible control functions}}, where $U\subset\mathbb{R}^{m}$
is a compact set, the \textbf{\textit{control range}}. The function
$F:M\times\mathbb{R}^{m}\rightarrow TM$ is a $C^{1}$-map such that for each
$u\in U$, $F_{u}(\cdot):=F(\cdot,u)$ is a smooth vector field on $M$. For each
$x\in M$ and $\omega\in\mathcal{U}$, we suppose that there exists an unique
solution $\varphi(t,x,\omega)$ which is defined for all $t\in\mathbb{R}$. We
usually refer to the solution $\varphi(\cdot,x,\omega)$ as a
\textbf{\textit{trajectory}} of $x$ with control function $\omega$ and write
$\varphi_{t}(x,\omega)=\varphi(t,x,\omega)$ where convenient.

We need several notions characterizing controllability properties of subsets
of the state space $M$ of system (\ref{2.1}).

For $x\in M$ and $t>0$, the \textbf{\textit{set of points reachable from}} $x$
\textbf{\textit{up to time}} $t$ and the \textbf{\textit{set of points
controllable to}} $x$ \textbf{\textit{within time}} $t$ are given by
\[
\mathcal{O}_{\leq t}^{+}(x):=\{y\in M;\ \mbox{ there are }s\in\lbrack
0,t]\mbox{ and }\omega\in\mathcal{U}\mbox{ with }\varphi(s,x,\omega)=y\},
\]
and
\[
\mathcal{O}_{\leq t}^{-}(x):=\{y\in M;\ \mbox{ there are }s\in\lbrack
0,t]\mbox{ and }\omega\in\mathcal{U}\mbox{ with }\varphi(s,y,\omega)=x\},
\]
respectively. The \textbf{\textit{positive}} and \textbf{\textit{negative
orbit from}} $x\in M$ are%
\[
\mathcal{O}^{+}(x):=\bigcup_{t>0}\mathcal{O}_{\leq t}^{+}(x)\text{ and
}\mathcal{O}^{-}(x):=\bigcup_{t>0}\mathcal{O}_{t}^{-}(x).
\]

A key concept of this paper is presented in the following definition.

\begin{definition}
A subset $D$ of $M$ is a \textbf{\textit{control set}} if

(i) for each $x\in D$, there exists $\omega\in\mathcal{U}$ with $\varphi
(\mathbb{R}_{+},x,\omega)\subset D$ (controlled invariance);

(ii) for each $x\in D$ one has $D\subset\overline{\mathcal{O}^{+}(x)}$
(approximate controllability);

(iii) $D$ is maximal with these properties.
\end{definition}

If for all $t>0$ the sets $\mathcal{O}_{\leq t}^{-}(x)$ and $\mathcal{O}_{\leq
t}^{+}(x)$ have nonempty interior, we say that system (\ref{2.1}) is
\textbf{\textit{locally accessible}} \textbf{\textit{from}} $x\in M$. Of main
interest are control sets with nonvoid interior which are locally accessible
from all $x\in\mbox{int}D$. Then $\mathrm{int}D\subset\mathcal{O}^{+}(x)$ for
all $x\in D$, cf. Colonius and Kliemann \cite[Lemma 3.2.13]{ColK00}.

Next we recall some basic concepts and results on linearization of a control system.

\begin{definition}
For a control-trajectory pair $(\omega(\cdot),\varphi(\cdot,x,\omega))$ the
linearized system is given by%
\begin{equation}
\frac{Dz}{dt}(t)=A(t)z(t)+B(t)\mu(t),\ \mu\in L^{\infty}(\mathbb{R}%
,\mathbb{R}^{m}), \label{linearization}%
\end{equation}
where $A(t):=\nabla F_{\omega(t)}(\varphi(t,x,\omega
))\ \mbox{ and }\ B(t):=D_{2}F(\varphi(t,x,\omega),\omega(t))$.
\end{definition}

The derivative on the left-hand side of (\ref{linearization}) is the covariant
derivative of $z(\cdot)$ along $\varphi(\cdot,x,\omega)$ and $D_{2}$ is the
derivative with respect to second component. A solution of
(\ref{linearization}) corresponding to $\mu\in L^{\infty}(\mathbb{R}%
,\mathbb{R}^{m})$ with initial value $\lambda\in T_{x}M$ is a locally
absolutely continuous vector field $z=\phi^{x,\omega}(\cdot,\lambda
,\mu):\mathbb{R}\rightarrow TM$ along $\varphi(\cdot,x,\omega)$ with
$z(0)=\lambda$, satisfying the differential equation (\ref{linearization}) for
almost all $t\in\mathbb{R}$.

The next proposition presents some properties of linearized systems.

\begin{proposition}
\label{p1} Let $(\omega(\cdot),\varphi(\cdot,x,\omega))$ be a
control-trajectory pair with corresponding linearization (\ref{linearization}%
). Then the following statements hold:

(i) For all $\tau>0$ the mapping $\varphi_{\tau}:M\times L^{\infty}%
([0,\tau],\mathbb{R}^{m})\rightarrow M,(x,\omega)\mapsto\varphi(\tau
,x,\omega)$ is continuously (Fr\'{e}chet) differentiable.

(ii) For every initial value $\lambda\in T_{x}M$ and every $\mu\in L^{\infty
}(\mathbb{R},\mathbb{R}^{m})$ there exists a unique solution $\phi^{x,\omega
}(\cdot,\lambda,\mu):\mathbb{R}\rightarrow TM$ of (\ref{linearization})
satisfying
\begin{equation}
\phi^{x,\omega}(0,\lambda,\mu)=\lambda,\phi^{x,\omega}(t,\lambda,\mu
)=D\varphi_{t}(x,\omega)(\lambda,\mu),t\in\mathbb{R}, \label{eq}%
\end{equation}
for $(\lambda,\mu)\in T_{x}M\times L^{\infty}(\mathbb{R},\mathbb{R}^{m})$,
where $D$ stands for the total derivative of $\varphi_{t}:M\times L^{\infty
}(\mathbb{R},\mathbb{R}^{m})\rightarrow M$ which consists of the derivative
$d_{x}\varphi_{t}(\cdot,\omega):T_{x}M\rightarrow T_{\varphi(t,x,\omega)}M$ in
the first, and the Fr\'{e}chet derivative of $\varphi_{t}(x,\cdot):L^{\infty
}(\mathbb{R},\mathbb{R}^{m})\rightarrow T_{\varphi(t,x,\omega)}M$ in the
second component.

(iii) For every $\tau>0$ the map $\phi^{x,\omega}(\tau,\cdot,\cdot
):T_{x}M\times L^{\infty}([0,\tau],\mathbb{R}^{m})\rightarrow T_{\varphi
(\tau,x,\omega)}M$ is linear and continuous.

(iv) For each $t\in\mathbb{R}$ abbreviate $\phi_{t}^{x,\omega}:=\phi
^{\varphi(t,x,\omega),\omega(t+\cdot)}.$ Then for all $t,s\in\mathbb{R}$,
$\lambda\in T_{x}M$ and $\mu\in L^{\infty}(\mathbb{R},\mathbb{R}^{m})$,
\[
\phi_{s}^{x,\omega}(t,\phi^{x,\omega}(s,\lambda,\mu),\Theta_{s}\mu
)=\phi^{x,\omega}(t+s,\lambda,\mu),
\]
and, in particular,
\[
\phi_{s}^{x,\omega}(t,\phi^{x,\omega}(s,\lambda,\mathbf{0}),\mathbf{0}%
)=\phi^{x,\omega}(t+s,\lambda,\mathbf{0}).
\]

\end{proposition}

The next definition introduces the notion of regularity of a
control-trajectory pair.

\begin{definition}
Consider some $(x,\omega,\tau)\in M\times\mathcal{U}\times(0,\infty)$ and let
$y:=\varphi(\tau,x,\omega)$. Then we call the linearization along
$(\omega(\cdot),\varphi(\cdot,x,\omega))$ \textbf{\textit{controllable on}}
$[0,\tau]$ if for each $\lambda_{1}\in T_{x}M$ and $\lambda_{2}\in T_{y}M$
there exists $\mu\in L^{\infty}([0,\tau],\mathbb{R}^{m})$ with
\[
\phi^{x,\omega}(\tau,\lambda_{1},\mu)=\lambda_{2}.
\]
In this case, we say that the control-trajectory pair $(\omega(\cdot
),\varphi(\cdot,x,\omega))$ is \textbf{\textit{regular}} on $[0,\tau]$.
\end{definition}

A control-trajectory pair $(\omega(\cdot),\varphi(\cdot,x,\omega))$ is called
$\tau$-periodic, $\tau\geq0$, if $(\varphi(t+\tau,x,\omega),\omega
(t+\tau))=(\varphi(t,x,\omega),\omega(t))$ for all $t\in\mathbb{R}$, or
equivalently if $\varphi(\tau,x,\omega)=x$ and $\Theta_{\tau}\omega=\omega$,
where $(\Theta_{\tau}\omega)(t)=\omega(t+\tau),t\in\mathbb{R}$, is the $\tau
$-shift on $\mathcal{U}$. A periodic regular control-trajectory pair enjoys
the property described in the following proposition (cf. \cite[Proposition
1.30]{Kawa13}).

\begin{proposition}
\label{p2} Let $(\omega(\cdot),\varphi(\cdot,x,\omega))$ be a $\tau$-periodic
control-trajectory pair which is regular on $[0,\tau]$. Then there exists
$C>0$ such that for every $\lambda\in T_{x}M$ there is $\mu\in L^{\infty
}([0,\tau],\mathbb{R}^{m})$ with $\phi^{x,\omega}(\tau,\lambda,\mu)=0_{x}$ and
$\Vert\mu\Vert_{\lbrack0,\tau]}\leq C|\lambda|$, where $\Vert\cdot
\Vert_{\lbrack0,\tau]}$ denotes the $L^{\infty}$-norm.
\end{proposition}

For a $\tau$-periodic control-trajectory pair $(\omega(\cdot),\varphi
(\cdot,x,\omega))$ the Floquet or Lyapunov exponents are given by%
\begin{equation}
\lim_{t\rightarrow\infty}\frac{1}{t}\log\left\Vert \phi^{x,\omega}%
(t,\lambda,\mathbf{0})\right\Vert =\lim_{n\rightarrow\infty}\frac{1}{n\tau
}\log\left\Vert \phi^{x,\omega}(n\tau,\lambda,\mathbf{0})\right\Vert
,\lambda\in T_{x}M.\label{exponents}%
\end{equation}
These limits exist and the Lyapunov exponents are denoted by $\rho_{1}%
(\omega,x),\dotsc,\rho_{r}(\omega,x)$ with $1\leq r:=r(\omega,x)\leq d=\dim
M$. The Lyapunov spaces are given by%
\[
L_{j}(\omega,x)=\left\{  \lambda\in T_{x}M;~\lim_{t\rightarrow\pm\infty}%
\frac{1}{t}\log\left\Vert \phi^{x,\omega}(t,\lambda,\mathbf{0})\right\Vert
=\rho_{j}(\omega,x)\right\}  ,j=1,\dotsc,r,
\]
with dimensions $d_{j}(\omega,x)$. They yield the decomposition%
\[
T_{x_{0}}M=L_{1}(\omega,x)\oplus\cdots\oplus L_{r}(\omega,x).
\]

\subsection{Invariance pressure}

In this subsection we recall the concepts of invariance and outer invariance
pressure introduced in Colonius, Cossich and Santana \cite{Cocosa, Cocosa2}
and some of their properties.

A pair $(K,Q)$ of nonempty subsets of $M$ is called
\textbf{\textit{admissible}} if $K$ is compact and for each $x\in K$ there
exists $\omega\in\mathcal{U}$ such that $\varphi(\mathbb{R}_{+},x,\omega
)\subset Q$. For an admissible pair $(K,Q)$ and $\tau>0$, a $(\tau
,K,Q)$-\textbf{\textit{spanning set}} $\mathcal{S}$ is a subset of
$\mathcal{U}$ such that for all $x\in K$ there is $\omega\in\mathcal{S}$ with
$\varphi(t,x,\omega)\in Q$ for all $t\in\lbrack0,\tau]$. Denote by
$C(U,\mathbb{R})$ the set of continuous function $f:U\rightarrow\mathbb{R}$
which we call \textbf{\textit{potentials}}.

For a potential $f\in C(U,\mathbb{R})$ denote $(S_{\tau}f)(\omega):=\int
_{0}^{\tau}f(\omega(t))dt$ and
\[
a_{\tau}(f,K,Q):=\inf\left\{  \sum_{\omega\in\mathcal{S}}e^{(S_{\tau}%
f)(\omega)};\ \mathcal{S}\text{ }(\tau,K,Q)\text{-spanning}\right\}  .
\]
The \textbf{\textit{invariance pressure}} $P_{inv}(f,K,Q)$ of control system
(\ref{2.1}) is defined by%
\[
P_{inv}(f,K,Q):=\underset{\tau\rightarrow\infty}{\lim\sup}\frac{1}{\tau}\log
a_{\tau}(f,K,Q).
\]
Given an admissible pair $(K,Q)$ such that $Q$ is closed in $M$, and a metric
$\varrho$ on $M$ which is compatible with the Riemannian structure, we define
the \textbf{\textit{outer invariance pressure}} of $(K,Q)$ by%
\[
P_{out}(f,K,Q):=\lim_{\varepsilon\rightarrow0}P_{inv}(f,K,N_{\varepsilon
}(Q)),
\]
where $N_{\varepsilon}(Q)=\{y\in M;\ \exists\ x\in Q\mbox{ with }\varrho
(x,y)<\varepsilon\}$ denotes the $\varepsilon$-neighborhood of $Q$.

Note that $P_{out}(f,K,Q)\leq P_{inv}(f,K,Q)\leq\infty$ for every admissible
pair $(K,Q)$ and all potentials $f$. For the potential $f=\mathbf{0}$, this
reduces to the notion of invariance entropy, $P_{inv}(\mathbf{0}%
,K,Q)=h_{inv}(K,Q)$ and $P_{out}(\mathbf{0},K,Q)=h_{out}(K,Q)$, cf. Kawan
\cite{Kawa13}.

The next proposition presents some properties of the function $P_{inv}%
(\cdot,K,Q):C(U,\mathbb{R})\rightarrow\mathbb{R}$, cf. \cite[Proposition
3.4]{Cocosa2}.

\begin{proposition}
\label{propert}The following assertions hold for an admissible pair $(K,Q)$,
functions $f,g\in C(U,\mathbb{R})$ and $c\in\mathbb{R}$:

(i) $P_{inv}(f,K,Q)\leq P_{inv}(g,K,Q)$ and $P_{out}(f,K,Q)\leq P_{out}%
(g,K,Q)$ for $f\leq g$.

(ii) $P_{inv}(f+c,K,Q)=P_{inv}(f,K,Q)+c$.

(iii) $h_{inv}(K,Q)+\min_{u\in U}f(u)\leq P_{inv}(f,K,Q)\leq h_{inv}%
(K,Q)+\max_{u\in U}f(u).$
\end{proposition}

\begin{remark}
\label{Remark_countable}If there is no countable $(\tau,K,Q)$-spanning set,
then $a_{\tau}(f,K,Q)=\infty$ (see Kawan \cite[Example 2.3]{Kawa13} for an
example). In particular, if $P_{inv}(f,K,Q)<\infty$, then for every $\tau>0$
there are countable $(\tau,K,Q)$-spanning sets. On the other hand if for all
$\tau>0$ there is a countable $(\tau,K,Q)$-spanning set, $a_{\tau
}(f,K,Q)=\infty$ is also possible. Proposition \ref{propert} (iii) shows that
$P_{inv}(f,K,Q)<\infty$ if and only if $h_{inv}(f,K,Q)<\infty$. If every
$(\tau,K,Q)$-spanning set $\mathcal{S}$ contains a finite $(\tau
,K,Q)$-spanning subset $\mathcal{S}^{\prime}$, then
\[
a_{\tau}(f,K,Q)=\inf\left\{  \sum_{\omega\in\mathcal{S}}e^{(S_{\tau}%
f)(\omega)};\ \mathcal{S}\text{ finite and }(\tau,K,Q)\text{-spanning}%
\right\}  .
\]
This follows, since all summands satisfy $e^{(S_{\tau}f)(\omega)}>0$, and
hence the summands in $\mathcal{S}\setminus\mathcal{S}^{\prime}$ can be
omitted. This situation occurs e.g. if $Q$ is open, where compactness of $K$
may be used. For the outer invariance entropy one considers $(\tau
,K,N_{\varepsilon}(Q))$-spanning sets, $\varepsilon>0$, and hence here it is
also sufficient to consider finite $(\tau,K,N_{\varepsilon}(Q))$-spanning
sets. For the inner invariance pressure of discrete time systems, one
considers sets which are $(\tau,K,\mathrm{int}Q)$-spanning. Here again finite
spanning sets are sufficient (the proof given in \cite[Propositon 5]{Cocosa}
for the case $K=Q$ easily extends to admissible pairs $(K,Q)$).
\end{remark}

\begin{remark}
The Lipschitz continuity property
\[
\left\vert P_{inv}(f,K,Q)-P_{inv}(g,K,Q)\right\vert \leq\left\Vert
f-g\right\Vert _{\infty}\text{ for }f,g\in C(U,\mathbb{R}),
\]
holds if $P_{inv}(f,K,Q),P_{inv}(g,K,Q)<\infty$. In fact, in this case, there
are for every $\tau>0$ countable $(\tau,K,Q)$-spanning sets $\mathcal{S}$ with
$\sum_{\omega\in\mathcal{S}}e^{(S_{\tau}f)(\omega)}<\infty$. Then the
arguments in \cite[Proposition 13(iii)]{Cocosa} can be applied in this
situation observing that the elementary lemma \cite[Lemma 12]{Cocosa}, on
which the proof is based, is valid not only for finite but also for infinite
sequences: Let $a_{i}\geq0,b_{i}>0,i\in\mathbb{N}$. Then for all
$n\in\mathbb{N}$%
\[
\frac{\sum_{i=1}^{n}a_{i}}{\sum_{i=1}^{n}b_{i}}\geq\min_{i=1,...n}\frac{a_{i}%
}{b_{i}}\geq\inf_{i\in\mathbb{N}}\frac{a_{i}}{b_{i}},
\]
and one may take the limit for $n\rightarrow\infty$.
\end{remark}

The following proposition shows that in the definition of invariance pressure
we can take the limit superior over times which are integer multiples of some
fixed time step $\tau>0$.

\begin{proposition}
\label{discretiza}The invariance pressure satisfies for every $\tau>0$%
\begin{equation}
P_{inv}(f,K,Q)=\limsup_{n\rightarrow\infty}\frac{1}{n\tau}\log a_{n\tau
}(f,K,Q)\text{ for all }f\in C(U,\mathbb{R}). \label{4.1b}%
\end{equation}

\end{proposition}

\begin{proof}
Let $(\tau_{k})_{k\geq1}$, $\tau_{k}\in(0,\infty)$ and $\tau_{k}%
\rightarrow\infty$. Then for every $k\geq1$ there exists $n_{k}\geq1$ such
that $n_{k}\tau\leq\tau_{k}\leq(n_{k}+1)\tau$ and $n_{k}\rightarrow\infty$ for
$k\rightarrow\infty$. Since $\tilde{f}(u):=f(u)-\inf f\geq0,u\in U$, it
follows that $a_{\tau_{k}}(\tilde{f},K,Q)\leq a_{(n_{k}+1)\tau}(\tilde
{f},K,Q)$ and consequently $\frac{1}{\tau_{k}}\log a_{\tau_{k}}(\tilde
{f},K,Q)$ is less than or equal to $\frac{1}{n_{k}\tau}\log a_{(n_{k}+1)\tau
}(\tilde{f},K,Q)$. Hence
\begin{align*}
\limsup_{k\rightarrow\infty}\frac{1}{\tau_{k}}\log a_{\tau_{k}}(\tilde
{f},K,Q)  &  \leq\limsup_{k\rightarrow\infty}\frac{1}{n_{k}\tau}\log
a_{(n_{k}+1)\tau}(\tilde{f},K,Q)\\
&  =\limsup_{k\rightarrow\infty}\frac{n_{k}+1}{n_{k}}\frac{1}{(n_{k}+1)\tau
}\log a_{(n_{k}+1)\tau}(\tilde{f},K,Q)\\
&  \leq\limsup_{n\rightarrow\infty}\frac{1}{n\tau}\log a_{n\tau}(\tilde
{f},K,Q).
\end{align*}
This shows that $P_{inv}(f-\inf f,K,Q)\leq\limsup_{n\rightarrow\infty}\frac
{1}{n\tau}\log a_{n\tau}(f-\inf f,K,Q)$, and by Proposition \ref{propert}
(iii) we obtain
\[
P_{inv}(f,K,Q)\leq\limsup_{n\rightarrow\infty}\frac{1}{n\tau}\log a_{n\tau
}(f,K,Q).
\]
The converse inequality is obvious.
\end{proof}

For the proof of the following proposition see \cite[Corollary 4.3]{Cocosa2}.

\begin{proposition}
\label{compact}Let $K_{1},K_{2}$ be two compact sets with nonempty interior
contained in a control set $D\subset M$. Then $(K_{1},Q)$ and $(K_{2},Q)$ are
admissible pairs and for all $f\in C(U,\mathbb{R})$ we have
\[
P_{inv}(f,K_{1},Q)=P_{inv}(f,K_{2},Q).
\]

\end{proposition}

\section{An upper bound on control sets\label{Section3}}

Our goal in this section is to obtain an upper bound for the invariance
pressure of a control set. We consider a smooth control system (\ref{2.1}) on
a Riemannian manifold $(M,g)$ under our standard assumptions.

In the following theorem, given a periodic control-trajectory pair
$(\omega(\cdot),\varphi(\cdot,x,\omega))$, the different Lyapunov exponents at
$(x,\omega)$ are denoted by $\rho_{1}(x,\omega),\dotsc,\rho_{r}(x,\omega
),r=r(x,\omega)$, with Lyapunov spaces of dimensions $d_{1}(x,\omega
),\dotsc,d_{r}(x,\omega)$, respectively.

\begin{theorem}
\label{theo}Let $D\subset M$ be a control set with nonempty interior and
compact closure for control system (\ref{2.1}). Then for every compact set
$K\subset D$ and every set $Q\supset D$, the pair $(K,Q)$ is admissible and
for all potentials $f\in C(U,\mathbb{R})$ the invariance pressure satisfies
\[
P_{inv}(f,K,Q)\leq\inf_{(T,x,\omega)}\left\{  \sum_{j=1}^{r(x,\omega)}%
\max\{0,d_{j}(x,\omega)\rho_{j}(x,\omega)\}+\frac{1}{T}\int_{0}^{T}%
f(\omega(s))ds\right\}  ,
\]
where the infimum is taken over all $(T,x,\omega)\in(0,\infty)\times
\mathrm{int}D\times\mathcal{U}$ such that the control-trajectory pair
$(\omega(\cdot),\varphi(\cdot,x,\omega))$ is $T$-periodic and regular and the
values $\omega(t),t\in\lbrack0,T]$, are in a compact subset of $\mathrm{int}U$.
\end{theorem}

\begin{remark}
For $f\equiv0$, the statement of the theorem reduces to Kawan \cite[Theorem
4.4]{Kawa11b},%
\[
h_{inv}(K,Q)=P_{inv}(\mathbf{0},K,Q)\leq\inf_{(T,x,\omega)}\left\{  \sum
_{j=1}^{r(x,\omega)}\max\{0,d_{j}(x,\omega)\rho_{j}(x,\omega)\}\right\}  .
\]

\end{remark}

\begin{proof}
The theorem will follow by inspection of the proof given in \cite[Theorem
4.4]{Kawa11b} for invariance entropy and by indicating the complementary
arguments needed for invariance pressure. Note first that by Proposition
\ref{compact} we can choose $K$ as an arbitrary compact subset of $D$ with
nonvoid interior. Let $(\omega_{0}(\cdot),\varphi(\cdot,x_{0},\omega_{0}))$ be
a $T$-periodic and regular control-trajectory pair as in the statement of the
theorem. Then fix real numbers $\varepsilon>0$ and
\[
S_{0}>\sum\limits_{j=1}^{r}\max(0,d_{j}\rho_{j})),
\]
where $d_{j}=d_{j}(x_{0},\omega_{0})$ and $\rho_{j}(x_{0},\omega
_{0}),j=1,\dotsc,r$. An ingenious and lengthy construction provides a compact
set $K=\mathrm{cl}(B_{b_{0}}(x_{0}))\subset D$ containing $x_{0}$ in the
interior with the following properties: For some $\tau=kT,k\in\mathbb{N}$, and
arbitrary $n\in\mathbb{N}$ one finds a set $\mathcal{S}_{n}$ of $(n\tau
,K,Q)$-spanning controls $\omega\in\mathcal{S}_{n}$ satisfying
\begin{equation}
\Vert\omega-\omega_{0}\Vert_{\lbrack0,n\tau]}\leq Cb_{0}\sqrt{d}%
,\label{Kawan4.17}%
\end{equation}
where $C>0$ is a constant and $b_{0}>0$ can be taken arbitrarily small (see
\cite[formula (4.17)]{Kawa11b}: the elements of $\mathcal{S}_{n}$ are $n$-fold
concatenations of the controls denoted there by $u_{x}$). The cardinality
$\#\mathcal{S}_{n}$ of $\mathcal{S}_{n}$ is bounded by
\begin{equation}
\frac{1}{n\tau}\log\#\mathcal{S}_{n}\leq S_{0}+\varepsilon,\label{Kawan_p745}%
\end{equation}
cf. \cite[estimate on middle of p. 745]{Kawa11b}.

In order to get a bound for the invariance pressure we need the following
additional arguments: Let $f\in C(U,\mathbb{R})$ be a potential. Since $f$ is
defined on the compact set $U$, its uniform continuity implies that there
exists $\delta>0$ such that $\left\Vert u-v\right\Vert <\delta$ implies
$|f(u)-f(v)|<\varepsilon$. Take $b_{0}>0$ small enough such that%
\[
Cb_{0}\sqrt{d}<\delta.
\]
By (\ref{Kawan4.17}) every $\omega\in\mathcal{S}_{n}$ satisfies $|\omega
(t)-\omega_{0}(t)|\leq\Vert\omega-\omega_{0}\Vert_{\lbrack0,n\tau]}<\delta$
for almost all $t\in\lbrack0,n\tau]$. Hence it follows that $|f(\omega
(t))-f(\omega_{0}(t))|<\varepsilon$ for almost all $t\in\lbrack0,n\tau]$.

Now we can estimate
\begin{align*}
\allowdisplaybreaks\frac{1}{n\tau}\log a_{n\tau}(f,K,Q) &  \leq\frac{1}{n\tau
}\log\sum_{\omega\in\mathcal{S}_{n}}e^{(S_{n\tau}f)(\omega)}=\frac{1}{n\tau
}\log\sum_{\omega\in\mathcal{S}_{n}}e^{\int_{0}^{n\tau}f(\omega(t))dt}\\
&  =\frac{1}{n\tau}\log\sum_{\omega\in\mathcal{S}_{n}}e^{\int_{0}^{n\tau
}f(\omega_{0}(t))dt+\int_{0}^{n\tau}[f(\omega(t))-f(\omega_{0}(t))]dt}\\
&  \leq\frac{1}{n\tau}\log\left[  \sum_{\omega\in\mathcal{S}_{n}}e^{\int
_{0}^{n\tau}f(\omega_{0}(t))dt}+\log e^{\int_{0}^{n\tau}\varepsilon
dt}\right]  \\
&  =\frac{1}{n\tau}\log\left(  \#\mathcal{S}_{n}e^{\int_{0}^{n\tau}%
f(\omega_{0}(t))dt}\right)  +\frac{1}{n\tau}\log e^{\int_{0}^{n\tau
}\varepsilon dt}\\
&  =\frac{1}{n\tau}\log\#\mathcal{S}_{n}+\frac{1}{n\tau}\int_{0}^{n\tau
}f(\omega_{0}(t))dt+\varepsilon\\
&  <S_{0}+\frac{1}{T}\int_{0}^{T}f(\omega_{0}(t))dt+2\varepsilon.
\end{align*}
For the last inequality we have used (\ref{Kawan_p745}) and $T$-periodicity of
$\omega_{0}$. By Proposition \ref{discretiza} this implies%
\[
P_{inv}(f,K,Q)=\underset{n\rightarrow\infty}{\lim\sup}\frac{1}{n\tau}\log
a_{n\tau}(f,K,Q)\leq S_{0}+\frac{1}{T}\int_{0}^{T}f(\omega_{0}%
(t))dt+2\varepsilon.
\]
Since $\varepsilon>0$ can be chosen arbitrarily small and $S_{0}$ arbitrarily
close to $\sum\limits_{j=1}^{r}\max(0,d_{j}\rho_{j})$, the assertion of the
theorem follows.
\end{proof}

\begin{remark}
In Kawan \cite[Section 5.2]{Kawa13} and da Silva and Kawan \cite[Section
3.2]{AdrKa} one finds more information on regular periodic control-trajectory pairs.
\end{remark}

\section{A lower bound\label{Section4}}

Again we consider a smooth control system (\ref{2.1}) on a Riemannian manifold
$(M,g)$ under our standard assumptions. Thus for each $t\geq0$ and each
control $\omega\in\mathcal{U}$ the map $\varphi_{t,\omega}:M\rightarrow M$ is
a diffeomorphism.

\begin{theorem}
\label{theorem_lower}Let $(K,Q)$ be an admissible pair where both $K$ and $Q$
have positive and finite volume. Then for every $f\in C(U,\mathbb{R})$%
\begin{align*}
&  P_{inv}(f,K,Q)\\
&  \geq\underset{\tau\rightarrow\infty}{\mathrm{lim~sup}}\frac{1}{\tau}\left(
\inf_{(x,\omega)}\int_{0}^{\tau}f(\omega(s))ds+\max\{0,\inf_{(x,\omega)}%
\int_{0}^{\tau}\mathrm{div}F_{\omega(s)}(\varphi(s,x,\omega))ds\}\right)  ,
\end{align*}
where both infima are taken over all $(x,\omega)\in K\times\mathcal{U}$ with
$\varphi([0,\tau],x,\omega)\subset Q$.
\end{theorem}

\begin{proof}
First observe that by Remark \ref{Remark_countable} we may assume that for all
$\tau>0$ there exists a countable $(\tau,K,Q)$-spanning set, since otherwise
$P_{inv}(f,K,Q)=\infty$, and the infimum in $a_{\tau}(f,K,Q)$ my be taken over
all countable $(\tau,K,Q)$-spanning sets $\mathcal{S}$. For each $\omega$ in a
countable $(\tau,K,Q)$-spanning set $\mathcal{S}$ define%
\[
K_{\omega}:=\{x\in K;\varphi([0,\tau],x,\omega)\subset Q\}.
\]
Thus $K=%
%TCIMACRO{\tbigcup \nolimits_{\omega\in\mathcal{S}}}%
%BeginExpansion
{\textstyle\bigcup\nolimits_{\omega\in\mathcal{S}}}
%EndExpansion
K_{\omega}$. Since $Q$ is Borel measurable, each set $K_{\omega}$ is
measurable as the countable intersection of measurable sets,%
\[
K_{\omega}=K\cap\bigcap_{t\in\lbrack0,\tau]\cap\mathbb{Q}}\varphi_{t,\omega
}^{-1}(Q).
\]
Then%
\begin{align*}
\mathrm{vol}(Q) &  \geq\mathrm{vol}(\varphi_{t,\omega}(K_{\omega}%
))=\int_{\varphi_{t,\omega}(K_{\omega})}\mathrm{dvol}=\int_{K_{\omega}%
}\left\vert \det\mathrm{d}_{x}\varphi_{t,\omega}\right\vert \mathrm{dvol}\\
&  \geq\mathrm{vol}(K_{\omega})\inf_{(x,\omega)}\left\vert \det\mathrm{d}%
_{x}\varphi_{t,\omega}\right\vert ,
\end{align*}
where the infimum is taken over all $(x,\omega)\in K\times\mathcal{U}$ with
$\varphi([0,\tau],x,\omega)\subset Q$. Abbreviating with the same infima
\[
\alpha(\tau):=\inf_{(x,\omega)}\left\vert \det\mathrm{d}_{x}\varphi_{t,\omega
}\right\vert ,~\beta(\tau):=\inf_{(x,\omega)}S_{\tau}(f)(\omega),
\]
we find%
\begin{align*}
e^{\beta(\tau)}\mathrm{vol}(K) &  \leq\sum_{\omega\in\mathcal{S}}e^{(S_{\tau
}f)(\omega)}\mathrm{vol}(K_{\omega})\leq\sup_{\omega\in\mathcal{S}%
}\mathrm{vol}(K_{\omega})\sum_{\omega\in\mathcal{S}}e^{(S_{\tau}f)(\omega)}\\
&  \leq\frac{\mathrm{vol}(Q)}{\max\{1,\alpha(\tau)\}}\sum_{\omega
\in\mathcal{S}}e^{(S_{\tau}f)(\omega)}.
\end{align*}
Since this holds for every countable $(\tau,K,Q)$-spanning set $\mathcal{S}$,
we find%
\begin{align*}
a_{\tau}(f,K,Q) &  =\inf\left\{  \sum_{\omega\in\mathcal{S}}e^{(S_{\tau
}f)(\omega)};\ \mathcal{S}\text{ countable }(\tau,K,Q)\text{-spanning}%
\right\}  \\
&  \geq\frac{\mathrm{vol}(K)}{\mathrm{vol}(Q)}e^{\beta(\tau)}\max
\{1,\alpha(\tau)\},
\end{align*}
implying%
\begin{align*}
P_{inv}(f,K,Q) &  =\underset{\tau\rightarrow\infty}{\mathrm{lim~sup}}\frac
{1}{\tau}\log a_{\tau}(f,K,Q)\geq\underset{\tau\rightarrow\infty
}{\mathrm{lim~sup}}\frac{1}{\tau}\left(  \beta(\tau)+\log\max\{1,\alpha
(\tau)\}\right)  \\
&  =\underset{\tau\rightarrow\infty}{\mathrm{lim~sup}}\frac{1}{\tau}\left(
\beta(\tau)+\max\{0,\log\alpha(\tau)\}\right)  .
\end{align*}
By Liouville's formula%
\begin{equation}
\log\det\mathrm{d}_{x}\varphi_{t,\omega}=\int_{0}^{\tau}\mathrm{div}%
F_{\omega(s)}(\varphi(s,x,\omega))ds,\label{Liouville}%
\end{equation}
and hence the assertion of the theorem follows:%
\begin{align*}
&  P_{inv}(f,K,Q)\\
&  \underset{\tau\rightarrow\infty}{\geq\mathrm{lim~sup}}\frac{1}{\tau}\left(
\inf_{(x,\omega)}\int_{0}^{\tau}f(\omega(s))ds+\max\{0,\inf_{(x,\omega)}%
\int_{0}^{\tau}\mathrm{div}F_{\omega(s)}(\varphi(s,x,\omega))ds\}\right)  .
\end{align*}

\end{proof}

\section{Linear control systems\label{Section5}}

In this section we prove a formula for the invariance pressure of linear
control systems in $\mathbb{R}^{d}$. They have the form
\begin{equation}
\dot{x}(t)=Ax(t)+B\omega(t),\ \omega\in\mathcal{U}, \label{lcs2}%
\end{equation}
where $A\in\mathbb{R}^{d\times d}$ and $B\in\mathbb{R}^{d\times m}$.

For system (\ref{lcs2}) there exists a unique control set $D$ with nonvoid
interior, if, without control constraint, the system is controllable and the
control range $U$ is a compact neighborhood of the origin in $\mathbb{R}^{m}$.
It is convex, and it is bounded if and only if $A$ is hyperbolic, i.e., there
is no eigenvalue of $A$ with vanishing imaginary part (cf. Hinrichsen and
Pritchard \cite[Theorems 6.2.22 and 6.2.23]{HiP18}, Colonius and Kliemann
\cite[Example 3.2.16]{ColK00}). The state space $\mathbb{R}^{d}$ can be
decomposed into the direct sum of the stable subspace $E^{s}$ and the unstable
subspace $E^{u}$ which are the direct sums of all generalized real eigenspaces
for the eigenvalues $\lambda$ with $\operatorname{Re}\lambda<0$ and
$\operatorname{Re}\lambda>0$, resp. Let $\pi:\mathbb{R}^{d}\rightarrow E^{u}$
be the projection along $E^{s}$. We obtain the following estimates.

\begin{lemma}
\label{theorem20}Consider a linear control system in $\mathbb{R}^{d}$ of the
form (\ref{lcs2}) and assume that the pair $(A,B)$ is controllable, that $A$
is hyperbolic and the control range $U$ is a compact neighborhood of the
origin. Let $D$ be the unique control set with nonvoid interior. Then for
every compact set $K\subset D$ with nonempty interior every potential $f\in
C(U,\mathbb{R})$ satisfies
\begin{align*}
&  \inf_{(T^{\prime},x^{\prime},\omega^{\prime})}\frac{1}{T^{\prime}}\int
_{0}^{T^{\prime}}f(\omega^{\prime}(s))ds\leq P_{inv}(f,K,D)-\sum_{j=1}%
^{r}d_{j}\max\{0,\operatorname{Re}\lambda_{j}\}\\
&  \leq\inf_{(T,x,\omega)}\frac{1}{T}\int_{0}^{T}f(\omega(s))ds,
\end{align*}
where the first infimum is taken over all $(T^{\prime},x^{\prime}%
,\omega^{\prime})\in\mathbb{R}_{+}\times\pi K\times\mathcal{U}$ with
$\pi\varphi([0,T^{\prime}],x^{\prime},\omega^{\prime})\allowbreak\subset\pi D$
and the second infimum is taken over all $(T,x,\omega)\in\lbrack
0,\infty)\times D\times\mathcal{U}$ such that the control-trajectory pair
$(\omega(\cdot),\varphi(\cdot,x,\omega))$ is $T$-periodic and contained in
$\mathrm{int}D$ and the values $\omega(t),t\in\lbrack0,\tau]$, are in a
compact subset of $\mathrm{int}U$.
\end{lemma}

\begin{proof}
The hypotheses imply (see \cite{ColK00}, Example 3.2.16) that $0\in
\mathrm{int}D\subset\mathbb{R}^{d}$ and the Lebesgue measure of $K$ and $D$
(which coincides with the volume) is finite and positive. Theorem \ref{theo}
yields%
\begin{equation}
P_{inv}(f,K,D)\leq\inf_{(T,x,\omega)}\left\{  \sum_{j=1}^{r(x,\omega)}%
\max\{0,d_{j}(x,\omega)\rho_{j}(x,\omega)\}+\frac{1}{T}\int_{0}^{T}%
f(\omega(s))ds\right\}  ,\label{linear_upper}%
\end{equation}
where the infimum is taken over all $T>0$ and all $(x,\omega)\in
\mathrm{int}D\times\mathcal{U}$ such that the control-trajectory pair
$(\omega(\cdot),\varphi(\cdot,x,\omega))$ is $T$-periodic and the values
$\omega(t),t\in\lbrack0,T]$, are in a compact subset of $\mathrm{int}U$. By
Floquet theory it follows (cf. \cite[Proposition 20]{Cocosa2}) that for all
$T$-periodic $(\omega(\cdot),\varphi(\cdot,x,\omega))$%
\[
\sum_{j=1}^{r(x,\omega)}\max\{0,d_{j}(x,\omega)\rho_{j}(x,\omega)\}=\sum
_{j=1}^{r}\max\{0,d_{j}\operatorname{Re}\lambda_{j}\},
\]
where the sum is over the $r$ eigenvalues $\lambda_{j}$ of $A$ with
multiplicities $d_{j}$. Hence
\begin{align*}
P_{inv}(f,K,D) &  \leq\inf_{(T,x,\omega)}\left\{  \sum_{j=1}^{r}\max
\{0,d_{j}\operatorname{Re}\lambda_{j}\}+\frac{1}{T}\int_{0}^{T}f(\omega
(s))ds\right\}  \\
&  \leq\inf_{(T,x,\omega)}\frac{1}{T}\int_{0}^{T}f(\omega(s))ds+\sum_{j=1}%
^{r}d_{j}\max\{0,\operatorname{Re}\lambda_{j}\},
\end{align*}
where the infimum is taken over all $(T,x,\omega)$ as in (\ref{linear_upper}).
This proves second inequality.

Hence it remains to prove the first inequality. By Theorem \ref{theorem_lower}%
\begin{align*}
&  P_{inv}(f,K,D)\\
&  \geq\underset{\tau\rightarrow\infty}{\mathrm{lim~sup}}\frac{1}{\tau}\left(
\inf_{(x,\omega)}\int_{0}^{\tau}f(\omega(s))ds+\max\left\{  0,\inf
_{(x,\omega)}\int_{0}^{\tau}\mathrm{div}F_{\omega(s)}(\varphi(s,x,\omega
))ds\right\}  \right) \\
&  \geq\inf_{(T,x,\omega)}\frac{1}{T}\int_{0}^{T}f(\omega(s))ds+\max\left\{
0,\inf_{(T,x,\omega)}\frac{1}{T}\int_{0}^{T}\mathrm{div}F_{\omega(s)}%
(\varphi(s,x,\omega))ds\right\}  ,
\end{align*}
where both infima in the second line are taken over all pairs $(x,\omega)\in
K\times\mathcal{U}$ with $\varphi([0,\tau],x,\omega)\subset D$ and both infima
in the third line are taken over all $(T,x,\omega)\in(0,\infty)\times
K\times\mathcal{U}$ with $\varphi([0,T],x,\omega)\subset D$. Liouville's
formula (\ref{Liouville}) (or direct inspection) shows
\[
\int_{0}^{\tau}\mathrm{div}F_{\omega(s)}(\varphi(s,x,\omega))ds=\log
\det\mathrm{d}_{x}\varphi_{t,\omega}=\sum_{j=1}^{r}d_{j}\operatorname{Re}%
\lambda_{j},
\]
where the sum is over the $r$ eigenvalues $\lambda_{j}$ of $A$ with
multiplicities $d_{j}$.

\textbf{Step 1:} Suppose that $\operatorname{Re}\lambda_{j}>0$ for all $j$.
Then%
\begin{align*}
&  P_{inv}(f,K,D)\\
&  \geq\inf_{(T,x,\omega)}\frac{1}{T}\int_{0}^{T}f(\omega(s))ds+\max\left\{
0,\inf_{(T,x,\omega)}\frac{1}{T}\int_{0}^{T}\mathrm{div}F_{\omega(s)}%
(\varphi(s,x,\omega))ds\right\} \\
&  =\inf_{(T,x,\omega)}\frac{1}{T}\int_{0}^{T}f(\omega(s))ds+\sum_{j=1}%
^{r}d_{j}\operatorname{Re}\lambda_{j},
\end{align*}
where the infimum is taken over all $(T,x,\omega)\in(0,\infty)\times
K\times\mathcal{U}$ with $\varphi([0,T],x,\omega)\subset D$.

\textbf{Step 2: }Next we treat the general case, where also eigenvalues with
negative real part are allowed. Recall that $\pi:\mathbb{R}^{d}\rightarrow
E^{u}$ denotes the projection onto the unstable subspace $E^{u}$ along the
stable subspace $E^{s}$.

Since these subspaces are $A$-invariant, this defines a semi-conjugacy between
system (\ref{lcs2}) and the system on $E^{u}$ given by%
\begin{equation}
\dot{y}(t)=A_{\left\vert E^{u}\right.  }y(t)+\pi Bu(t),u\in\mathcal{U},
\label{lin_proj}%
\end{equation}
with trajectories $\pi\varphi(\cdot,x^{\prime},\omega^{\prime})$, and the sets
$K$ and $D$ are mapped to $\pi K$ and $\pi D$, resp. Then $\pi K$ and $\pi D$
have positive volume and form an admissible pair (cf. Kawan \cite[proof of
Theorem 3.1]{Kawa13}) . One easily proves that (cf. \cite[Proposition
10]{Cocosa})%
\[
P_{inv}(f,K,Q)\geq P_{inv}(f,\pi K,\pi Q),
\]
since every $(\tau,K,D)$-spanning set yields a $(\tau,\pi K,\pi D)$-spanning
set. Similarly as in Step 1, Theorem \ref{theorem_lower} applied to system
(\ref{lin_proj}) implies that%
\[
P_{inv}(f,\pi K,\pi D)\geq\inf_{(T^{\prime},x^{\prime},\omega^{\prime})}%
\frac{1}{T^{\prime}}\int_{0}^{T^{\prime}}f(\omega^{\prime}(s))ds+\sum
_{j=1}^{r}d_{j}\max\{0,\operatorname{Re}\lambda_{j}\},
\]
where the infimum is taken over all $(T^{\prime},x^{\prime},\omega^{\prime
})\in\mathbb{R}_{+}\times\pi K\times\mathcal{U}$ with $\pi\varphi
([0,T^{\prime}],x^{\prime},\omega^{\prime})\subset\pi D$.
\end{proof}

Next we show that the two infima in the proposition above actually coincide
and again use hyperbolicity of $A$ in a crucial way. This provides the
announced formula for the invariance pressure.

\begin{theorem}
\label{linear_main2}Consider a linear control system in $\mathbb{R}^{d}$ of
the form (\ref{lcs2}) and assume that the pair $(A,B)$ is controllable, the
matrix $A$ is hyperbolic and the control range $U$ is a compact neighborhood
of the origin. Let $D$ be the unique control set with nonvoid interior. Then
for every compact set $K\subset D$ with nonempty interior every potential
$f\in C(U,\mathbb{R})$ satisfies
\begin{equation}
P_{inv}(f,K,D)=\min_{u\in U}f(u)+\sum_{j=1}^{r}d_{j}\max\{0,\operatorname{Re}%
\lambda_{j}\}. \label{conjecture}%
\end{equation}

\end{theorem}

\begin{proof}
Let $\varepsilon>0$ and consider a control $\omega_{0}\in\mathcal{U}$
satisfying%
\begin{align*}
\frac{1}{T_{0}}\int_{0}^{T_{0}}f(\omega_{0}(s))ds &  \leq\inf_{(T^{\prime
},\omega^{\prime})\in(0,\infty)\times\mathcal{U}}\frac{1}{T^{\prime}}\int
_{0}^{T^{\prime}}f(\omega^{\prime}(s))ds+\varepsilon\\
&  \leq\inf_{(T^{\prime},x^{\prime},\omega^{\prime})}\frac{1}{T^{\prime}}%
\int_{0}^{T^{\prime}}f(\omega^{\prime}(s))ds+\varepsilon,
\end{align*}
where the second infimum is taken over all triples $(T^{\prime},x^{\prime
},\omega^{\prime})\in\mathbb{R}_{+}\times\pi K\times\mathcal{U}$ with
$\pi\varphi([0,T],x^{\prime},\omega^{\prime})\subset\pi D$. Observe that there
is a control value $u_{0}$ with%
\[
f(u_{0})=\min_{u\in U}f(u)=\inf_{(T^{\prime},\omega^{\prime})\in
(0,\infty)\times\mathcal{U}}\frac{1}{T^{\prime}}\int_{0}^{T^{\prime}}%
f(\omega^{\prime}(s))ds.
\]
Then there is a control%
\[
\omega_{1}\in\mathrm{int}\mathcal{U}_{\left\vert [0,T_{0}]\right.  }=\{u\in
L^{\infty}([0,T_{0}];u(t),t\in\lbrack0,T_{0}]\text{, in a compact subset of
}\mathrm{int}U\}
\]
with
\[
\frac{1}{T_{0}}\int_{0}^{T_{0}}f(\omega_{1}(s))ds\leq\frac{1}{T_{0}}\int
_{0}^{T_{0}}f(\omega_{0}(s))ds+\varepsilon.
\]
\textbf{Claim:} For every $T>0$ and every control $\omega\in\mathcal{U}$ there
exists $x_{1}\in\mathbb{R}^{d}$ with $\varphi(T,x_{1},\omega)=x_{1}$.

In fact, hyperbolicity of $A$ implies that the matrix $I-e^{AT}$ is
invertible, and hence there is a unique solution $x(T,\omega)$ of%
\[
\left(  I-e^{AT}\right)  x(T,\omega)=\varphi(T,0,\omega).
\]
Now the variation-of-constants formula shows the claim:%
\[
x(T,\omega)=e^{AT}x(T,\omega)+\varphi(T,0,\omega)=\varphi(T,x(T,\omega
),0)+\varphi(T,0,\omega)=\varphi(T,x(T,\omega),\omega).
\]
Applying this to $T_{0}$ and $\omega_{1}$ we find a point $x_{1}%
:=x(T_{0},\omega_{1})=\varphi(T_{0},x_{1},\omega_{1})$. Since $\omega_{1}%
\in\mathrm{int}\mathcal{U}_{\left\vert [0,T_{0}]\right.  }$it follows that a
neighborhood of $x_{1}$ can be reached in time $T_{0}$ from $x_{1}$. This
follows, since by controllability, the map%
\[
L_{\infty}([0,T_{0}],\mathbb{R}^{m})\rightarrow\mathbb{R}^{d},\omega
\mapsto\varphi(T_{0},0,\omega)
\]
is a linear surjective map, hence maps open sets to open sets, and the same is
true for the map%
\[
\omega\mapsto\varphi(T_{0},x_{1},\omega)=e^{AT_{1}}+\varphi(T_{0},0,\omega).
\]
Analogously, $x_{1}$ can be reached from every point in a neighborhood of
$x_{1}$ in time $T_{0}$. Hence in the intersection of these two neighborhoods
every point can be steered in time $2T_{0}$ into every other point. This shows
that $x_{1}$ is in the interior of the (unique) control set $D$, and the
corresponding trajectory $\varphi(t,x_{1},\omega_{1}),t\in\lbrack0,T_{0}]$,
remains in the interior of $D$. Extending $\omega_{1}(t),t\in\lbrack0,T_{0}]$,
to a $T_{0}$-periodic control $\omega_{2}$ we find that the control-trajectory
pair $(\omega_{2}(\cdot),\varphi(\cdot,x_{1},\omega_{2}))$ is $T_{0}%
$-periodic, the trajectory is contained in $\mathrm{int}D$ and the values
$\omega_{2}(t+T_{0})=\omega_{1}(t),t\in\lbrack0,T_{0}]$, are in a compact
subset of $\mathrm{int}U$. It follows that%
\begin{align*}
\inf_{(T^{\prime},\omega^{\prime})}\frac{1}{T^{\prime}}\int_{0}^{T^{\prime
\prime}}f(\omega^{\prime}(s))ds &  \geq\frac{1}{T_{0}}\int_{0}^{T_{0}}%
f(\omega_{0}(s))ds-\varepsilon\geq\frac{1}{T_{0}}\int_{0}^{T_{0}}f(\omega
_{1}(s))ds-2\varepsilon\\
&  =\frac{1}{T_{0}}\int_{0}^{T_{0}}f(\omega_{2}(s))ds-2\varepsilon\\
&  \geq\inf_{(T,x,\omega)}\frac{1}{T}\int_{0}^{T}f(\omega(s))ds-2\varepsilon,
\end{align*}
where the second infimum is taken over all $(T,x,\omega)\in(0,\infty)\times
D\times\mathcal{U}$ such that the control-trajectory pair $(\omega
(\cdot),\varphi(\cdot,x,\omega))$ is $T$-periodic, the trajectory is contained
in $\mathrm{int}D$ and the values $\omega(t),t\in\lbrack0,T]$, are in a
compact subset of $\mathrm{int}U$.

Together with the inequalities in Lemma \ref{theorem20} this implies%
\begin{align*}
\min_{u\in U}f(u) &  =\inf_{(T^{\prime},\omega^{\prime})}\frac{1}{T^{\prime}%
}\int_{0}^{T^{\prime\prime}}f(\omega^{\prime}(s))ds\leq\inf_{(T^{\prime
},y,\omega^{\prime})}\frac{1}{T^{\prime}}\int_{0}^{T^{\prime}}f(\omega
^{\prime}(s))ds\\
&  \leq P_{inv}(f,K,D)-\sum_{j=1}^{r}d_{j}\max\{0,\operatorname{Re}\lambda
_{j}\}\\
&  \leq\inf_{(T,x,\omega)}\frac{1}{T}\int_{0}^{T}f(\omega(s))ds\\
&  \leq\inf_{(T^{\prime},\omega^{\prime})}\frac{1}{T^{\prime}}\int
_{0}^{T^{\prime\prime}}f(\omega^{\prime}(s))ds+2\varepsilon\\
&  =\min_{u\in U}f(u)+2\varepsilon.
\end{align*}
Since $\varepsilon>0$ is arbitrary, assertion (\ref{conjecture}) follows.
\end{proof}

\begin{remark}
The proof of the \textbf{Claim} above follows arguments in the proof of da
Silva and Kawan \cite[Theorem 20]{daSilK18}.
\end{remark}

\begin{remark}
\label{Remark_old}Theorem \ref{linear_main2} improves \cite[Theorem
6.2]{Cocosa2}, where it had to be assumed that the minimum of $f(u),u\in U$,
is attained in an equilibrium.
\end{remark}

\section{Further applications\label{Section6}}

In this section, we apply Theorem \ref{theo} to linear control systems on Lie
groups and to inner control sets.

\subsection{Control sets and equilibrium pairs}

Given a control system (\ref{2.1}), a pair $(u_{0},x_{0})\in U\times M$ is
called an \textbf{\textit{equilibrium pair}} if $F(x_{0},u_{0})=0$, or
equivalently, $\varphi(t,x_{0},\bar{u}_{0})=x_{0}$ for all $t\in\mathbb{R}$,
where $\bar{u}_{0}(t)\equiv u_{0}$.

If $(u_{0},x_{0})$ is an equilibrium pair, the linearized system is an
autonomous linear control system in $T_{x_{0}}M$ and the Lyapunov exponents at
$(u_{0},x_{0})$ in the direction $\lambda\in T_{x_{0}}M\backslash\{0_{x_{0}%
}\}$ coincide with the real parts of the eigenvalues of $\nabla F_{u_{0}%
}(x_{0}):T_{x_{0}}M\rightarrow T_{x_{0}}M$. Then regularity, i.e.,
controllability of the linearized system, can be checked by Kalman's rank condition.

\begin{corollary}
\label{corol1} Let $D\subset M$ be a control set with nonempty interior and
let $f\in C(U,\mathbb{R})$. Suppose that there is a regular equilibrium pair
$(u_{0},x_{0})\in\mathrm{int}U\times\mathrm{int}D$. Then for every compact set
$K\subset D$ and every set $Q\supset D$ we have
\[
P_{inv}(f,K,Q)\leq\sum_{\lambda\in\sigma(\nabla F_{u_{0}}(x_{0}))}%
\max\{0,n_{\lambda}\emph{Re}(\lambda)\}+f(u_{0}),
\]
where $n_{\lambda}$ is the algebraic multiplicity of the eigenvalue $\lambda$
in the spectrum $\sigma(\nabla F_{u_{0}}(x_{0}))$.
\end{corollary}

\begin{proof}
Since $(u_{0},x_{0})$ is a regular equilibrium pair, the control-trajectory
pair $(\varphi(\cdot,x_{0},\bar{u}_{0}),\bar{u}_{0}(\cdot))$ is $T$-periodic
and regular for every $T>0$. By Theorem \ref{theo} we obtain
\begin{align*}
P_{inv}(f,K,Q) &  \leq\inf_{(T,x,\omega)}\left\{  \sum_{j=1}^{r(x,\omega)}%
\max\{0,d_{j}(x,\omega)\rho_{j}(x,\omega)\}\right\}  +\frac{1}{T}\int_{0}%
^{T}f(\omega(s))ds\\
&  \leq\sum_{\lambda\in\sigma(\nabla F_{\omega_{0}}(x_{0}))}\max
\{0,n_{\lambda}\emph{Re}(\lambda)\}+f(u_{0}).
\end{align*}

\end{proof}

\subsection{Control sets of linear control systems on Lie groups}

In this subsection we consider \textbf{\textit{linear control systems on a
connected Lie group }}$G$ introduced in Ayala and San Martin \cite{AySan} and
Ayala and Tirao \cite{AyTio}.

They are given by a family of ordinary differential equations on $G$ of the
form
\begin{equation}
\dot{x}(t)=\mathcal{X}(x(t))+\sum_{j=1}^{m}u_{j}(t)X_{j}(x(t)),\ \omega
=(u_{1},\dotsc,u_{m})\in\mathcal{U},\label{lcs1}%
\end{equation}
where the drift vector field $\mathcal{X}$, called the \textit{\textbf{linear
vector field}}, is an infinitesimal automorphism, i.e., its solutions are a
family of automorphisms of the group, and the $X_{j}$ are right invariant
vector fields. Note that the linear control systems of the form (\ref{lcs2})
are a special case with $G=$ $\mathbb{R}^{d}$.

Their controllability properties have been analyzed in da Silva \cite{daSil16}%
, Ayala, da Silva and Zsigmond \cite{ASZ} and Ayala and da Silva \cite{AydaS}.
In particular, the existence and uniqueness of control sets for general
systems of the form (\ref{lcs1}) has been analyzed in \cite{ASZ}. If $0$ is in
the interior of the control range $U$ and the reachable set $\mathcal{O}%
^{+}(e)$ from the neutral element $e$ is open (this holds e.g. if
$e\in\mathrm{int}\mathcal{O}^{+}(e)$), then there exists a control set $D$
containing $e$ in the interior. For semisimple or nilpotent Lie groups $G$
sufficient conditions for boundedness of $C$ are given in \cite[Theorem
3.9]{ASZ}, and \cite[Corollary 3.12]{ASZ} shows uniqueness of the control set
with nonvoid interior if $G$ is decomposable.

Along with system (\ref{lcs1}) comes an associated derivation $\mathcal{D}$ of
the Lie algebra $\mathfrak{g}$ of $G$ which is given by
\[
\mathcal{D}(Y)=-\mbox{ad}(\mathcal{X})(Y):=[\mathcal{X},Y](e).
\]

\begin{corollary}
\label{corol2}Consider the linear control system (\ref{lcs1}) on a Lie group
$G$. Suppose that $D$ is a control set with $e_{G}\in\mathrm{int}D$ and
compact closure $\overline{D}$ and let $K\subset D\subset Q$. Let $f\in
C(U,\mathbb{R})$ be a potential. If the equilibrium pair $(0,e_{G}%
)\in\mathrm{int}U\times\mathrm{int}D$ is regular, then
\[
P_{inv}(f,K,Q)\leq\sum_{\lambda\in\sigma(\mathcal{D})}\max\{0,n_{\lambda
}\emph{Re}(\lambda)\}+f(u_{0}).
\]
If furthermore $K$ has positive Haar measure and $f(0)=\min_{u\in U}f$, then
\[
P_{inv}(f,K,Q)=P_{out}(f,K,Q)=\sum_{\lambda\in\sigma(\mathcal{D})}%
\max\{0,n_{\lambda}\emph{Re}(\lambda)\}+f(0).
\]

\end{corollary}

\begin{proof}
Note that the right hand side of the system is given by $F(x,u)=\mathcal{X}%
(x)+\sum_{i=1}^{m}u_{i}X_{i}(x)$ and hence $F_{0}(x):=F(x,0)=\mathcal{X}(x)$.
Let $g$ the Riemannian metric on $G$ defined as in the proof of Theorem
\ref{theo} and $\nabla$ the Levi-Civita connection. Let $(\phi,U)$ be a local
coordinate neighborhood of $e_{G}$ and pick a left invariant vector field $Y$
in the Lie algebra $\mathfrak{g}$ of $G$. Then we can express $\mathcal{X}$ in
terms of $(\phi,U)$ by
\[
\mathcal{X}(h)=\sum_{i=1}^{d}y_{i}(h)\frac{\partial}{\partial x_{i}}.
\]
Note that since $\mathcal{X}(e_{G})=0$, then $y_{i}(e_{G})=0$ for every
$i\in\{1,\cdots,d\}$, hence
\[
(\nabla_{\mathcal{X}}Y)(e_{G})=\sum_{i=1}^{d}y_{i}(e_{G})\left(  \nabla
_{\frac{\partial}{\partial x_{i}}}Y\right)  (e_{G})=0.
\]
Since $\nabla$ is symmetric, we have
\[
\left(  \nabla_{Y}F_{0}\right)  (e_{G})=\left(  \nabla_{Y}\mathcal{X}\right)
(e_{G})=\left(  \nabla_{\mathcal{X}}Y-[\mathcal{X},Y]\right)  (e_{G}%
)=-[\mathcal{X},Y]=\mathcal{D}(Y).
\]
Since this holds for every $Y\in\mathfrak{g}$, we have $\nabla F_{0}%
(e_{G})=\mathcal{D}$. By Corollary \ref{corol1} we obtain
\[
P_{inv}(f,K,Q)\leq\sum_{\lambda\in\sigma(\mathcal{D})}\max\{0,n_{\lambda
}\operatorname{Re}(\lambda)\}+f(0).
\]
Now, suppose that $K$ has positive Haar measure. By da Silva \cite[Theorem
4.3]{daSil14}, we know that
\[
h_{out}(K,Q)\geq\sum_{\lambda\in\sigma(\mathcal{D})}\max\{0,n_{\lambda
}\operatorname{Re}(\lambda)\}.
\]
Define $\tilde{f}(u)=f(u)-f(u_{0}),u\in U$. Since $\tilde{f}\geq0$ Proposition
\ref{propert}(i) implies that
\[
P_{inv}(\tilde{f},K,Q)\geq P_{out}(\tilde{f},K,Q)\geq h_{out}(K,Q)\geq
\sum_{\lambda\in\sigma(\mathcal{D})}\max\{0,n_{\lambda}\operatorname{Re}%
(\lambda)\}.
\]
Proposition \ref{propert}(ii) implies $P_{inv}(\tilde{f},K,Q)=P_{inv}%
(f,K,Q)-\inf f$, hence this yields
\[
P_{inv}(f,K,Q)=P_{out}(f,K,Q)=\sum_{\lambda\in\sigma(\mathcal{D})}%
\max\{0,n_{\lambda}\operatorname{Re}(\lambda)\}+\inf f.
\]

\end{proof}

\subsection{Inner control sets}

This section presents an application of Theorem \ref{theo} to the class of
\textbf{\textit{inner control sets}} as defined (with small changes) in Kawan
\cite[Definition 2.6]{Kawa13}. This nomenclature refers to a control set
$D\subset M$ for which there exists an decreasing family of compact and convex
sets $\{U_{\rho}\}_{\rho\in\lbrack0,1]}$ in $\mathbb{R}^{m}$ (i.e.,
$U_{\rho_{2}}\subset U_{\rho_{1}}$ for $\rho_{1}<\rho_{2}$), such that for
every $\rho\in\lbrack0,1]$ system (\ref{2.1})$_{\rho}$ with control range
$U_{\rho}$ (instead of $U$ in (\ref{2.1})) has a control set $D_{\rho}$ with
nonvoid interior and compact closure, and the following conditions are satisfied:

(i) $U=U_{0}$ and $D=D_{1}$;

(ii) $\overline{D_{\rho_{2}}}\subset\mathrm{int}D_{\rho_{1}}$ whenever
$\rho_{1}<\rho_{2}$;

(iii) for every neighborhood $W$ of $\overline{D}$ there is $\rho\in
\lbrack0,1)$ with $\overline{D_{\rho}}\subset W$.

\begin{corollary}
Consider an inner control set $D$ of control system (\ref{2.1}). Let
$(\omega_{0}(\cdot),\varphi(\cdot,x_{0},\omega_{0}))$ be a regular
$T$-periodic control-trajectory pair with $x_{0}\in\overline{D}$ and
$\omega_{0}\in\mathcal{U}_{1}$. Then
\[
P_{out}(f,\overline{D})\leq\sum_{j=1}^{r}\max\{0,d_{j}\operatorname{Re}%
\lambda_{j}\}+\frac{1}{T}\int_{0}^{T}f(\omega_{0}(s))ds,
\]
holds, where $\lambda_{1},\dotsc,\lambda_{r}$ are the Lyapunov exponents at
$(x_{0},\omega_{0})$ with corresponding multiplicities $d_{1},\dotsc,d_{r}$.
\end{corollary}

\begin{proof}
Note that the definition of inner control sets implies that for every $\rho
\in\lbrack0,1)$ the set $\overline{D}$ is a compact subset of $D_{\rho}$. By
Theorem \ref{theo} it follows that the outer invariance pressure
$P_{out}^{\rho}(f,\overline{D},\overline{D_{\rho}})$ for system (\ref{2.1}%
)$_{\rho}$ satisfies%
\[
P_{out}^{\rho}(f,\overline{D},\overline{D_{\rho}})\leq\sum_{j=1}^{r}%
\max\{0,d_{j}\rho_{j}\}+\frac{1}{T}\int_{0}^{T}f(\omega_{0}(s))ds\ \text{for
all }\rho\in\lbrack0,1).
\]
Now for given $\varepsilon>0$ we may choose $\rho\in\lbrack0,1)$ such that
$\overline{D_{\rho}}\subset N_{\varepsilon}(\overline{D})$. Then
\begin{align*}
P_{out}(f,\overline{D},N_{\varepsilon}(\overline{D})) &  \leq P_{out}^{\rho
}(f,\overline{D},N_{\varepsilon}(\overline{D}))\\
&  \leq P_{out}^{\rho}(f,\overline{D_{\rho}},N_{\varepsilon}(\overline{D}))\\
&  \leq\sum_{j=1}^{r}\max\{0,d_{j}\rho_{j}\}+\frac{1}{T}\int_{0}^{T}%
f(\omega_{0}(s))ds.
\end{align*}
The first two inequalities follow from $U_{\rho}\subset U_{0}$ and
$\overline{D_{\rho}}\subset N_{\varepsilon}(\overline{D})$. Since
$P_{out}(f,\overline{D})=\lim_{\varepsilon\rightarrow0}P_{out}(f,\overline
{D},N_{\varepsilon}(\overline{D}))$, the assertion follows.
\end{proof}

\subsection{Example}

Consider the following linear control system in $\mathbb{R}^{d}$,
\[
\left[
\begin{array}
[c]{c}%
\dot{x}\\
\dot{y}%
\end{array}
\right]  =\underbrace{\left[
\begin{array}
[c]{cc}%
1 & -1\\
1 & 1
\end{array}
\right]  }_{=:A}\left[
\begin{array}
[c]{c}%
x\\
y
\end{array}
\right]  +\underbrace{\left[
\begin{array}
[c]{c}%
0\\
1
\end{array}
\right]  }_{=:B}\omega(t)
\]
and assume that $\omega(t)\in U:=[-1,1]+u_{0}$ for some $u_{0}\in(-1,1)$. In
this case, $0\in\mathrm{int}U$ and $(A,B)$ is controllable, and $A$ is
hyperbolic with eigenvalues given by $\lambda_{\pm}=1\pm i$. There exists a
unique control set $D\subset\mathbb{R}^{2}$ such that $(0,0)\in\mathrm{int}D$,
and $\overline{D}$ is compact.

We may interpret the control functions $\omega(t)$ and also $u_{0}$ as
external forces acting on the system. Take $f\in C(U,\mathbb{R})$ as
$f(u):=|u-u_{0}|$, then $(S_{\tau}f)(\omega)$ represents the impulse of
$\omega-u_{0}$ until time $\tau$. For a subset $K\subset D$ a $(\tau
,K,D)$-spanning set $\mathcal{S}$ represents a set of external forces $\omega$
that cause the system to remain in $D$ when it starts in $K$. By Theorem
\ref{linear_main2} we obtain for a compact subset $K\subset D$ with nonzero
Lebesgue measure that
\[
P_{inv}(f,K,Q)=2+\min_{u\in U}f(u)=2+\min_{u\in\lbrack-1,1]+u_{0}}\left\vert
u-u_{0}\right\vert .
\]
Here $P_{inv}(f,K,Q)$ represent the exponential growth rate of the amount of
total impulse required of the external forces $\omega-u_{0}$ acting on the
system to remain in $D$ as time tends to infinity. The minimum of $f$ is
attained in $u=u_{0}$, which does not correspond to an equilibrium if
$u_{0}\not =0$. Hence \cite[Theorem 6.2]{Cocosa2} (cf. Remark \ref{Remark_old}%
) could not be applied in this case.

\end{document}